\documentclass[twoside,leqno,10pt]{amsart}
\usepackage{amsfonts}
\usepackage{amsmath}
\usepackage{amscd}
\usepackage{amssymb}
\usepackage{amsthm}
\usepackage{latexsym}
\usepackage{bbm}
\numberwithin{equation}{section}

\begin{document}
\newtheorem{theorem}{Theorem}
\newtheorem{lemma}{Lemma}
\newtheorem{prop}{Proposition}
\newtheorem{corollary}{Corollary}
\newtheorem{conjecture}{Conjecture}
\numberwithin{equation}{section}
\newcommand{\dif}{\mathrm{d}}
\newcommand{\intz}{\mathbb{Z}}
\newcommand{\ratq}{\mathbb{Q}}
\newcommand{\natn}{\mathbb{N}}
\newcommand{\comc}{\mathbb{C}}
\newcommand{\rear}{\mathbb{R}}
\newcommand{\prip}{\mathbb{P}}
\newcommand{\uph}{\mathbb{H}}
\newcommand{\fie}{\mathbb{F}}

\title{A Remark on the Conjectures of Lang-Trotter and Sato-Tate on Average}
\date{\today}
\author{Stephan Baier}
\maketitle

\begin{abstract} 
We obtain new average results on the conjectures of Lang-Trotter and Sato-Tate about elliptic curves. 
\end{abstract}



\section{The conjectures of Sato-Tate and Lang-Trotter}
Before we state our results, we first explain briefly the contents of the conjectures of Lang-Trotter and Sato-Tate on elliptic curves.

Let $E$ be an elliptic curve over $\ratq$. For any prime number $p$ of good reduction, let $a_p(E)$ be the trace of the Frobenius morphism of $E/\fie_p$. Then the number of points on the reduced curve modulo $p$ equals $\#E(\fie_p)= p+1- a_p(E)$. Furthermore, by Hasse's theorem,
$|a_p(E)|\le 2\sqrt{p}$.

For the case that $E$ does not have complex multiplication, Sato and Tate \cite{Tate} formulated a conjecture on the distribution of angles associated to the numbers $a_p(E)$ which is equivalent to the following assertion on the distribution of the $a_p(E)$'s.\\ \\
{\bf Sato-Tate Conjecture:} \begin{it}
Suppose $E$ is an elliptic curve over $\ratq$ which does not admit complex multiplication. For any $-1\le \alpha< \beta\le 1$, and $x\ge 1$, let
\[ \Theta_E(\alpha,\beta;x):=\sum\limits_{\substack{p\le x\\ \alpha\le a_p(E)/(2\sqrt{p})\le\beta}}\log p. \]
Then 
\[ \lim\limits_{x\rightarrow\infty} \frac{\Theta_E(\alpha,\beta;x)}{x} =\frac{2}{\pi}\int\limits_{\alpha}^{\beta} \sqrt{1-t^2}\ {\rm d}t. \] \end{it}

In \cite{CHT}, \cite{HSBT} and \cite{Tayl}, L. Clozel, M. Harris, N. Shepherd-Barron and R. Taylor have proved the Sato-Tate conjecture for all elliptic curves $E$ over totally real fields (in particular, over $\mathbbm{Q}$) satisfying the mild condition of having multiplicative reduction at some prime. 

Lang and Trotter \cite{LTr} considered the quantity
$$
\pi_E^r(x):= \#\{p\leq x: a_p(E)=r\},
$$
where $r$ is a fixed integer.
If $r=0$ and $E$ has complex multiplication, Deuring \cite{Deu} showed that
\begin{equation} \label{deu}
\pi_E^0(x)\sim \frac{\pi(x)}{2} \ \ \ \ \mbox{ as } x\rightarrow\infty.
\end{equation}
For all other cases, Lang and Trotter \cite{LTr} conjectured that the following asymptotic
estimate holds.\\ \\
{\bf Lang-Trotter Conjecture:} \begin{it} If $E$ has no complex multiplication or $r\not=0$, we have
$$
\pi_E^r(x)\sim C_{E,r}\pi_{1/2}(x), \ \ \ \
\mbox{ as } x\to\infty,
$$
where $C_{E,r}$ is some non-negative constant depending on $E$ and $r$.\end{it}\\ \\ 
Here, as in the sequel, 
$$
\pi_{1/2}(x):=\int\limits_{2}^x \frac{{\rm d}t}{2\sqrt{t}\log t}.
$$
Lang and Trotter \cite{LTr} used a probabilistic
model to
give an explicit description of the constant $C_{E,r}$ as an Euler product. 
The constant can be
$0$, and the asymptotic estimate is then interpreted to mean that there is
only a finite number of primes such that $a_p(E)=r$. 

\section{Main results}
In the sequel, by $E(a,b)$ we denote an elliptic curve given in Weierstrass form
$$
y^2=x^3+ax+b.
$$
As in \cite{Bai} and \cite{DaP}, we define a constant $C_r$ by
$$
C_r:=\frac{2}{\pi}\prod\limits_{p|r} \left(1-\frac{1}{p^2}\right)^{-1}
\prod\limits_{p\nmid r} \frac{p(p^2-p-1)}{(p-1)(p^2-1)}.
$$
In particular, if $r=0$, we have $$C_0=\frac{2\zeta(2)}{\pi}=\frac{\pi}{3}.$$ 
We shall establish the following average estimate of Lang-Trotter-type.\\

\begin{theorem} \label{LT}
Let $\varepsilon>0$ and $C>3/2+\varepsilon$ be given. Fix an integer $r\not=0$. Then, if
\begin{equation} \label{cond1}
A,B> x^{\varepsilon} \ \ \ \ \mbox{ and }\ \ \ \ 
x^{3/2+\varepsilon}<AB<x^{C},
\end{equation}
we have, as $x\rightarrow\infty$,
\begin{equation} \label{Cor}
\frac{1}{4AB}\sum\limits_{|a|\le A} \sum\limits_{|b|\le B} \pi_{E(a,b)}^r(x)\sim
C_r\pi_{1/2}(x).
\end{equation}
If $r=0$, then, under the conditions in \eqref{cond1}, we have, as $x\rightarrow\infty$,
\begin{equation} \label{Coro}
\frac{1}{4AB}\sum\limits_{1\le |a|\le A}\ \sum\limits_{1\le |b|\le B} \pi_{E(a,b)}^0(x)\sim
\frac{\pi}{3}\pi_{1/2}(x).
\end{equation}
\end{theorem}\medskip

In \eqref{Coro}, we have excluded the elliptic curves in the families $E(a,0)$ and $E(0,b)$ with $a,b\not=0$ because it turns out that 
$$
\frac{1}{4AB} \left(\sum\limits_{1\le |a|\le A} \pi_{E(a,0)}^0(x)+ 
\sum\limits_{1\le |b|\le B} \pi_{E(0,b)}^0(x)\right) \gg \left(\frac{1}{A}+\frac{1}{B}\right) \pi(x)
$$
which is much larger than $\frac{\pi}{3}\pi_{1/2}(x)$ if $A,B$ are small compared to $\sqrt{x}$. 
This is due to the fact that the curves in the said families have complex multiplication in which case we have Deuring's result \eqref{deu}. 

All other curves with complex multiplication are of the form $E_{(\alpha_it^2,\beta_it^3)}$, where $t\in \mathbbm{Z}/\{0\}$, and $(\alpha_i,\beta_i)$ is in an explicit set of eleven pairs of integers (see \cite{FMu}, page 3, for example). Hence, if $AB> x^{3/2+\varepsilon}$, their contribution to \eqref{Coro} is 
$$
\ll \frac{\min\{A^{1/2},B^{1/3}\}}{AB} \pi(x) \ll \frac{x}{\sqrt{AB}} \ll x^{1/4}
$$
which is negligible compared to the main term.

In \cite{Bai}, we proved Theorem \ref{LT} under the conditions
$$A,B>x^{1/2+\varepsilon}\ \ \ \ \mbox{ and }\ \ \ \ AB>x^{3/2+\varepsilon}$$  
in place of \eqref{cond1}. Hence, unlike the corresponding Theorem 2 in \cite{Bai}, the above Theorem \ref{LT} applies to situations when $A$ and $B$ are very small compared to $x^{1/2}$. Our additional condition $AB<x^{C}$ in Theorem \ref{LT} is not a real constraint since we are mainly interested in averages for small $A$'s and $B$'s, and it is likely that this condition can be removed by a refined treatment of a certain error term in section 3. 

We further note that the above-mentioned Theorem 2 in \cite{Bai} in turn was a generalization of an average result by E. Fouvry and M.R. Murty \cite{FMu} on $\pi_{E(a,b)}^0(x)$ and an improvement of a result of C. David and F. Pappalardi \cite{DaP} who showed the asymptotic formula \eqref{Cor} under the stronger condition $A,B>x^{1+\varepsilon}$.

Moreover, we shall prove the following average result on the Sato-Tate conjecture.

\begin{theorem} \label{ST}
Let $\varepsilon,c>0$ and $C>3/2+2\varepsilon$ be given. 
Further, let $x\ge 1$ and $0< \alpha< \beta\le 1$. Set 
\begin{equation}\label{Fdef}
F(\alpha,\beta):=\frac{2}{\pi}
\int\limits_{\alpha}^{\beta} \sqrt{1-t^2}\ {\rm d}t \ \ \ \ \mbox{ and } \ \ \ \
\gamma:=\beta-\alpha. 
\end{equation}
Assume that 
$x^{\varepsilon-5/12}\le \gamma/\beta \le x^{-\varepsilon}$ and $F(\alpha,
\beta)\ge x^{-1/2+\varepsilon}$.
Then, if 
\begin{equation} \label{cond3}
A,B>x^{\varepsilon}\ \ \ \ \mbox{ and }\ \ \ \  x^{1+\varepsilon}/F(\alpha,\beta)<AB<x^{C},
\end{equation}
we have
\begin{equation} \label{stest}
 \frac{1}{4AB}\sum\limits_{1\le |a|\le A} \sum\limits_{1\le |b|\le B} \Theta_{E(a,b)}(\alpha,\beta;x) =xF(\alpha,\beta) \left(1+O\left(\frac{1}{\log^c x}\right)\right),
\end{equation}
where the implied $O$-constant depends only on $\varepsilon$, $c$ and $C$.
\end{theorem}

To avoid technical complications, we have excluded the cases when $ab=0$. This makes sense because, as mentioned above, all elliptic curves $E(a,0)$ and $E(0,b)$ ($a,b\not=0$) have complex multiplication, and the Sato-Tate conjecture is exclusively formulated for curves {\it without} complex multiplication (if $E$ has complex multiplication, the distribution of the $a_p(E)$'s is different from the Sato-Tate distribution). 

We recall that the number of all remaining curves with complex multiplication is $O(\min\{A^{1/2},B^{1/3}\})$. Hence, if $AB> x^{1+\varepsilon}/F(\alpha,\beta)$, their contribution to \eqref{stest} is, by a trivial estimation,
$$
\ll \frac{\min\{A^{1/2},B^{1/3}\}}{AB}x \ll \frac{x}{\sqrt{AB}} \ll x^{1/2}
$$
which is majorized by the error term $xF(\alpha,\beta)/\log^c x$ since we assume that $F(\alpha,\beta)\ge x^{-1/2+\varepsilon}$.
 
In \cite{BaZh}, L. Zhao and I proved Theorem \ref{ST} (with the cases when $ab=0$ included)
under the conditions 
$$
A,B>x^{1/2+\varepsilon}\ \ \ \ \mbox{ and }\ \ \ \  AB>x^{1+\varepsilon}/F(\alpha,\beta)
$$
in place of \eqref{cond3}.
Again, Theorem \ref{ST} in the present paper allows much more flexibility in the choice of $A$ and $B$, and the condition $AB<x^{C}$ therein is not a real constraint.

From Theorem \ref{ST}, we derive the following corollary on
the Sato-Tate conjecture on average for {\it fixed} $\alpha$ and $\beta$.

\begin{corollary} \label{STcor}
Let $\varepsilon,c>0$ and $C>1+\varepsilon$ be given, and let $\alpha$, $\beta$ be fixed real numbers with $0< \alpha< \beta< 1$. Define $F(\alpha,\beta)$ as in Theorem \ref{ST}.
Then, if 
\begin{equation} \label{cond2}
A,B> x^{\varepsilon}\ \ \ \  \mbox{ and }\ \ \ \  x^{1+\varepsilon}<AB<x^{C},
\end{equation}
we have, as $x\rightarrow \infty$,
\begin{equation} \label{asy}
\frac{1}{4AB}\sum\limits_{1\le |a|\le A} \sum\limits_{1\le |b|\le B} \Theta_{E(a,b)}(\alpha,\beta;x) =xF(\alpha,\beta) \left(1+O\left(\frac{1}{\log^c x}\right)\right), 
\end{equation}
where the implied $O$-constant depends only on $\alpha$, $\beta$, $\varepsilon$, $c$ and $C$.
\end{corollary}

\begin{proof} 
If $0< \alpha< \beta < 1$ and $x$ is sufficiently large, then it is possible to split
the interval $[\alpha,\beta]$ into a finite number of subintervals $[\alpha',\beta']$ satisfying $x^{-2\varepsilon/3}\le (\beta'-\alpha')/\beta'=\gamma'/\beta' \le x^{-\varepsilon/3}$ and $F(\alpha',
\beta')\ge x^{-2\varepsilon/3}$. Now applying Theorem \ref{ST} with $\varepsilon$ replaced by
$\varepsilon/3$ to each of these subintervals, and summing up all contributions, we obtain the desired asymptotic estimate \eqref{asy} under the conditions in \eqref{cond2}. 
\end{proof}

We note that Theorem 14 in the recent work \cite{WDBIES} of W. Banks and I.E. Shparlinski implies the asymptotic estimate \eqref{asy} as well (the contributions of $a,b$ with $ab=0$ is negligible), but they require the  
conditions 
$$x^{\varepsilon}<A,B\le x^{1-\varepsilon} \ \ \ \ \mbox{ and } \ \ \ \  AB>x^{1+\varepsilon}\sqrt{\min\{A,B\}}$$
which are stronger than our conditions in \eqref{cond2}.

On the other hand, their error term estimate is {\it uniform} with respect to $\alpha$ and $\beta$, unlike that in our Corollary \ref{STcor}, and their estimate is sharper than ours by a factor of $x^{-\delta}$. Moreover, their result is valid for all $\alpha$, $\beta$ with $-1\le \alpha<\beta\le 1$ (in fact, they consider angles corresponding to $\alpha$ and $\beta$, which lie in the interval $[0,\pi]$). 
Our method certainly works for $\alpha,\beta$ with $-1<\alpha<\beta<0$ as well, but so far it doesn't cover intervals $[\alpha,\beta]$ containing $-1$, $0$ or $1$.

We note that the work of L. Clozel, M. Harris, N. Shepherd-Barron and R. Taylor in \cite{CHT}, \cite{HSBT} and \cite{Tayl} on the Sato-Tate conjecture for individual elliptic curves does not imply any of the above average results due to the lack of uniformity of the error term with respect to $a$ and $b$, and due to the lack of sufficiently strong zero density estimates for symmetric power $L$-functions. (Such zero density estimates would be required to establish a version of the Sato-Tate conjecture on individual elliptic curves for {\it small} intervals $[\alpha,\beta]$.)
 
We achieve our improvements by employing an almost-all result on character sums by Banks and Shparlinski which played an important rule in their paper \cite{WDBIES} too and is a consequence of a more general result by Garaev \cite{gar}. This result turns out to be more useful in the estimation of certain error terms than the bound of Polya-Vinogradov, which we used in \cite{Bai} and \cite{BaZh} at corresponding places, since it applies to very short character sums.

\section{Proof of Theorem 1}
In the following, let $r\not= 0$. We first estimate the contribution of all elliptic curves in the families $E(a,0)$ and $E(0,b)$ ($a,b\not=0$). We again note that these curves have complex multiplication. Further, if $E$ is an elliptic curve with complex multiplication, then, with an absolute $\ll$-constant not depending on $E$ or $r$, we have the bound
$$
\pi_E^r(x)\ll x^{1/2}.
$$
This is due to the fact that if $E$ has complex multiplication and $r\not=0$, then the primes $p$ satisfying $a_p(E)=r$ are of the form $p=f_{E,r}(n)/4$, where $n$ is an integer and $f_{E,r}$ is a certain quadratic polynomial with integer coefficients (see the equation and inequality before Theorem 9 in \cite{Co}). It follows that if $r\not=0$, then
\begin{equation} \label{0}
\frac{1}{4AB}\left(\sum\limits_{1\le |a|\le A} \pi_{E(a,0)}^r(x) + \sum\limits_{1\le |b|\le B} \pi_{E(0,b)}^r(x)\right)\ll \left(\frac{1}{A}+\frac{1}{B}\right)x^{1/2} \ll x^{1/2-\varepsilon}.
\end{equation}

It remains to estimate the sum 
$$
\sum\limits_{1\le |a|\le A} \sum\limits_{1\le |b|\le B} \pi_{E(a,b)}^r(x),
$$
where we now admit all integers $r$ (including $r=0$).
Here we follow our method in \cite{Bai}, with the alteration that we use a result due to Banks, Shparlinski and Garaev instead of the Polya-Vinogradov estimate to bound a certain error term. We shall be brief at all places where we don't alter the method in \cite{Bai}.

Similarly as in equation (2.2) in \cite{Bai}, the quantity in question can be written in the form
\begin{equation} \label{rew}
\sum\limits_{1\le |a|\le A} \sum\limits_{1\le |b|\le B} \pi_{E(a,b)}^r(x)=
\sum\limits_{B(r)<p\le x} \sharp\{1\le |a|\le A,\ 1\le |b|\le B\ :\ a_p(E(a,b))=r\},
\end{equation}
where $B(r)=\max\{3,r,r^2/4\}$. In \cite{Bai}, we first estimated the contribution of $a$'s and $b$'s with $p|ab$
by
\begin{equation} \label{negl}
\ll \frac{AB}{p}+A+B,
\end{equation}
which turned out to be a small enough error term. We then evaluated the remaining term
$$
\sharp\{|a|\le A,\ |b|\le B\ :\ p\nmid ab,\ a_p(E(a,b))=r\}.
$$
In the present note, the bound \eqref{negl} is not sufficient due to the fact that we admit $A$'s and $B$'s that are much smaller than in \cite{Bai}. In the following, we establish a refined estimate for the contribution in question.
We observe that 
\begin{eqnarray} \label{new}
& & \sum\limits_{B(r)<p\le x} \sharp\{1\le |a|\le A, 1\le |b|\le B\ :\ p|ab,\ a_p(E(a,b))=r\}\\
&\le &  4\ \sum\limits_{1\le a\le A}\ \sum\limits_{1\le b\le B}\ \sum\limits_{p|ab} \ 1\nonumber\\
&\le & 4\ \sum\limits_{n\le AB} \tau(n)^2 \nonumber\\
&\ll & (AB)^{1+\varepsilon_0}\nonumber
\end{eqnarray}
for every fixed $\varepsilon_0>0$,
where $\tau(n)$ is the number of divisors of $n$. By \eqref{new} and our condition $AB<x^C$ in Theorem \ref{ST}, the above contribution is indeed negligible if $C<1/(2\varepsilon_0)$.

The remaining term is
$$
\sum\limits_{B(r)<p\le x} \sharp\{|a|\le A,\ |b|\le B\ :\ p\nmid ab,\ a_p(E(a,b))=r\},
$$
which we shall evaluate in the following.
By Lemma 1 in \cite{Bai} (see also Lemma 1 in \cite{BaZh}) due to Deuring, the total number of $\fie_p$-isomorphism classes of elliptic curves
$E(c,d)$ over $\fie_p$ with $p+1-r$ points equals the Kronecker class number $H(r^2-4p)$. 
Let $I_{r,p}$ be the number of $\fie_p$-isomorphism classes of elliptic curves
$E(c,d)$ over $\fie_p$ with $p+1-r$ points
such that $c,d\not=0$. Hence,
\begin{equation} \label{triv}
I_{r,p}\le H(r^2-4p). 
\end{equation} 
Let $(u_{p,j},v_{p,j})$, $j=1,...,I_{r,p}$ be pairs of
integers such that the curves $E(\overline{u_{p,j}},\overline{v_{p,j}})$
form a
system of representatives of these isomorphism classes, where $\overline{n}$ denotes the reduction of an integer $n$ modulo $p$. 
Let $(\cdot/p)_4$ 
be the biquadratic residue symbol. Then, as observed in section 4 in \cite{Bai}, if $p\equiv 1$ mod $4$, we have 
\begin{eqnarray} \label{char}
& & \sharp\{|a|\le A,\ |b|\le B\ :\ p\nmid ab,\ a_p(E(a,b))=r\}\\  &=&
\frac{1}{4\varphi(p)} \sum\limits_{k=1}^4
\sum\limits_{\chi\ \! \mbox{\scriptsize mod } p} \sum\limits_{j=1}^{I_{r,p}}
\left(\frac{u_{p,j}}{p}\right)_4^{-k} \overline{\chi}^3(u_{p,j})
\chi^2(v_{p,j})
\sum\limits_{|a|\le A}\left(\frac{a}{p}\right)_4^{k}\chi^3(a)
\sum\limits_{|b|\le B}  \overline{\chi}^2(b)\nonumber
\\ 
&=& M(p)+E_1(p)+E_2(p),\nonumber
\end{eqnarray}
where\\ 

$M(p)=$ contribution of $k,\chi$ with
$(\cdot/p)_4^{k} \chi^3=\chi_0$,
$\chi^2=\chi_0$;\medskip

$E_1(p)=$ contribution of $k,\chi$ with
$(\cdot/p)_4^{k} \chi^3\not=\chi_0$,
$\chi^2=\chi_0$ or $(\cdot/p)_4^{k} \chi^3=\chi_0$,
$\chi^2\not=\chi_0$;\medskip

$E_2(p)=$ contribution of $k,\chi$ with
$(\cdot/p)_4^{k} \chi^3\not=\chi_0$, $\chi^2\not=\chi_0$.\\ \\
As noted in \cite{Bai}, in the case $p\equiv 3$ mod $4$, a similar representation of the term $$
\sharp\{|a|\le A,\ |b|\le B\ :\ p\nmid ab,\ a_p(E(a,b))=r\}
$$
as a character sum is possible, and this expression can be treated in a similar way as the above expression in the case $p\equiv 1$ mod $4$. Therefore, as in \cite{Bai}, we can confine ourselves to primes $p$ with $p\equiv 1$ mod $4$.
 
In \cite{Bai} we used results in \cite{DaP} to treat the the main term $M(p)$. The error term $E_1(p)$ was estimated by using the Polya-Vinogradov inequality, and the error term $E_2(p)$ was handled by the Cauchy-Schwarz inequality and some mean value estimates for character sums. Our estimate for $E_1(p)$ gave rise to the condition $A,B\ge x^{1/2+\varepsilon}$ in
Theorem 2 in \cite{Bai}, and our estimate for $E_2(p)$ gave rise to the condition $AB\ge x^{3/2+\varepsilon}$ in the same theorem. 

In the following, we want to refine the estimation of 
$$
\sum\limits_{\substack{B(r)<p\le x\\ p\equiv 1\ \mbox{\scriptsize \rm mod}\ 4}} |E_1(p)|
$$
by using the following variant of an almost-all result on character sums of Banks and Shparlinski \cite {WDBIES} which is contained in a more general result, Theorem 10 in \cite{gar}, of Garaev.

\begin{lemma} \label{gara} Fix $\varepsilon>0$ and $\eta>0$. If $x>0$ is sufficiently large,
then for all $M\ge x^{\varepsilon}$, all primes with at most $x^{3/4+4\eta+o(1)}$ exceptions,
and all non-principal multiplicative characters $\chi$ modulo $p$, we have
$$
\left|\sum\limits_{|n|\le M} \chi(n)\right| \le M^{1-\eta},
$$
where the function implied by $o(1)$ depends only on $\varepsilon$ and $\eta$.  
\end{lemma}

\begin{proof} This is Lemma 3 in \cite{WDBIES}, except that there the above sum is replaced by 
$$
\sum\limits_{n=1}^M \chi(n).
$$ 
But
$$
\sum\limits_{|n|\le M} \chi(n) = (1+\chi(-1))\sum\limits_{n=1}^M \chi(n)
$$
and hence, the result follows.
\end{proof}

We also need the following bound for $I_{r,p}$.

\begin{lemma} \label{Ilemma} If $|r|< 2\sqrt{p}$, then
$$I_{r,p}\le H(r^2-4p) \ll p^{1/2}\log^2 p.$$
\end{lemma}
\begin{proof} The inequality $I_{r,p}\le H(r^2-4p)$ was stated in \eqref{triv}. By Lemma 3 in \cite{BaZh}, for the Kronecker class number $H(r^2-4p)$ we have the formula
$$
H(r^2-4p)=\frac{1}{\pi}\sum\limits_{\substack{f,d\\ r^2-4p=df^2\\ d\equiv 0,1\
\bmod{4}}} \sqrt{|d|}L(1,\chi_d),
$$
where $\chi_d$ is a certain real character with conductor $\ll d$ (the above formula follows from a relation between the Kronecker and Dirichlet class numbers, and the Dirichlet class number formula). Now using the well-known bound
$$
L(1,\chi_d)\ll \log d,
$$
the desired result follows by a quick computation.
\end{proof}

As noted in \cite{Bai}, for each $k$ the number of characters $\chi$ modulo $p$ satisfying $(\cdot/p)_4^{k} \chi^3\not=\chi_0$,
$\chi^2=\chi_0$ or $(\cdot/p)_4^{k} \chi^3=\chi_0$,
$\chi^2\not=\chi_0$ is bounded. Therefore, Lemma \ref{Ilemma} implies the bound
\begin{equation}\label{fir}
|E_1(p)|\ll \frac{\log^2 p}{p^{1/2}} \left(A
\max\limits_{\substack{\chi\ \! \mbox{\scriptsize mod } p\\ \chi\not=\chi_0}} 
\left|\sum\limits_{|b|\le B}  \chi(b)\right|+B
\max\limits_{\substack{\chi\ \! \mbox{\scriptsize mod } p\\ \chi\not=\chi_0}} 
\left|
\sum\limits_{|a|\le A}\chi(a)\right|\right).
\end{equation}
From \eqref{fir} and Lemma \ref{gara} with $\eta=1/20$, we now obtain 
\begin{equation} \label{E11}
\sum\limits_{\substack{B(r)<p\le x\\ p\equiv 1\ \mbox{\scriptsize \rm mod}\ 4}} |E_1(p)|
\ll x^{1/2}(\log x)^2 AB^{19/20}+x^{1/2}(\log x)^2A^{19/20}B+x^{9/20+\varepsilon}AB.
\end{equation}
Similarly, one can prove that
\begin{equation} \label{E12}
\sum\limits_{\substack{B(r)<p\le x\\ p\equiv 3\ \mbox{\scriptsize \rm mod}\ 4}} |E_1(p)|
\ll x^{1/2}(\log x)^2 AB^{19/20}+x^{1/2}(\log x)^2A^{19/20}B+x^{9/20+\varepsilon}AB.
\end{equation}

Moreover, from the first equation after (4.3) in \cite{Bai}, Lemma 3 in \cite{Bai}, and Lemma \ref{Ilemma} above, we deduce that 
\begin{equation} \label{M}
\sum\limits_{B(r)<p\le x} M(p)
= 4C_r\pi_{1/2}(x)AB+ O\left(\frac{AB\sqrt{x}}{\log^c x}\right)
\end{equation}
for any given $c>0$,
and from the first inequality after (4.4) in \cite{Bai} and Lemma \ref{Ilemma} above, we deduce that
\begin{equation} \label{E2}
\sum\limits_{B(r)<p\le x} |E_2(p)|
\ll x^{5/4}(\log x)^4(AB)^{1/2}.
\end{equation}
Now, combining  \eqref{rew}, \eqref{new}, \eqref{char}, \eqref{E11}, \eqref{E12}, \eqref{M} and \eqref{E2},
we obtain the estimate
\begin{eqnarray} \label{LTfin}
& &\frac{1}{4AB}\sum\limits_{1\le |a|\le A} \sum\limits_{1\le |b|\le B}\pi_{E(a,b)}^r(x)\\
&=&
C_r\pi_{1/2}(x)+O\left((AB)^{\varepsilon_0-1}+x^{1/2}(\log x)^2\left(\frac{1}{A^{1/20}}+\frac{1}{B^{1/20}}+\frac{1}{x^{1/20-\varepsilon}}\right) +\frac{x^{5/4}\log^4 x}{\sqrt{AB}}+
\frac{\sqrt{x}}{\log^c x}\right).\nonumber
\end{eqnarray}
From \eqref{0} and \eqref{LTfin}, we deduce that
the desired asymptotic estimates \eqref{Cor} and \eqref{Coro} hold 
under the conditions in \eqref{cond1}. This completes the proof of Theorem 1.

\section{Proof of Theorem 2} 
We follow our method in \cite{BaZh}, with the alteration that we again use Lemma \ref{gara} instead of the Polya-Vinogradov estimate to bound a certain error term. Since we proceed similarly as in the previous section, we shall be very brief. Similarly as in \cite{BaZh} and in the previous section, we first write the quantity
$$ 
\sum\limits_{1\le |a|\le A} \sum\limits_{1\le |b|\le B} \Theta_{E(a,b)}(\alpha,\beta;x)
$$
in question as a character sum $\mathcal{X}$ plus some error term which can be bounded in a similar way as in \eqref{new} and is negligible under the condition $AB<x^{C}$ with $C>3/2+2\varepsilon$ being arbitrarily given. We then split our character sum $\mathcal{X}$ into a main term of the form 
$$\mathcal{M}=\sum\limits_{p\le x} M(p)$$
and two error terms $\mathcal{E}_1$, $\mathcal{E}_2$ of the form  
$$\mathcal{E}_i=\sum\limits_{p\le x} E_i(p).$$ We don't change our treatments of the main term and the second error term in \cite{BaZh} at all. To bound these terms, we required the conditions $F(\alpha,\beta)\ge x^{-1/2+\varepsilon}$, $x^{\varepsilon-5/12}\le \gamma/\beta \le x^{-\varepsilon}$ and $AB>x^{1+\varepsilon}/F(\alpha,\beta)$ in \cite{BaZh}. The treatment of the first error term in \cite{BaZh} led to the additional condition $A,B>x^{1/2+\varepsilon}$ which we aim to replace by $A,B>x^{\varepsilon}$. To this end, we need to estimate this error term $\mathcal{E}_1$ by a different technique. 

For the proof of Theorem \ref{ST} it now suffices to establish that
\begin{equation} \label{aim}
\frac{1}{4AB} |\mathcal{E}_1| \ll  F(\alpha,\beta)\frac{x}{\log^c x}
\end{equation}
holds for every fixed $c>0$ if $A,B>x^{\varepsilon}$.
By the considerations in \cite{BaZh}, if $p\equiv 1$ mod $4$, then $E_1(p)$ is of the form
\begin{equation} \label{char2}
E_1(p)=
\frac{1}{4\varphi(p)} \sum\limits_{k=1}^4\ \sideset{}{'}\sum\limits_{\chi\ \! \bmod{p}}\ \sum\limits_{j=1}^{I_{p}} \left(\frac{u_{p,j}}{p}\right)_4^{-k} \overline{\chi}^3(u_{p,j}) \chi^2(v_{p,j}) \sum\limits_{|a|\le A}\left(\frac{a}{p}\right)_4^{k}\chi^3(a) \sum\limits_{|b|\le B}  \overline{\chi}^2(b),
\end{equation}
where the sum $\sideset{}{'}\sum\limits_{\chi\ \! \bmod{p}}$ is taken over all characters such that
$(\cdot/p)_4^{k} \chi^3\not=\chi_0$, $\chi^2=\chi_0$ or $(\cdot/p)_4^{k} \chi^3=\chi_0$,
$\chi^2\not=\chi_0$, the number $I_p$ satisfies the bound
\begin{equation} \label{Ip}
I_p\le \sum\limits_{2\sqrt{p}\alpha \le r\le 2\sqrt{p}\beta} H(r^2-4p) =:H_p,
\end{equation}
and $u_{p,j},v_{p,j}$ are certain integers.
By (5.4) in \cite{BaZh} and our condition $F(\alpha,\beta)\ge x^{-1/2+\varepsilon}$, we have the bound
\begin{equation} \label{Hp}
H_p\ll x^{1+\varepsilon_1}F(\alpha,\beta)
\end{equation}
for any fixed $\varepsilon_1>0$, the implied $\ll$-constant depending only on $\varepsilon_1$. Using \eqref{char2}, \eqref{Ip}, \eqref{Hp} and the fact that the number of summands of the sum $\sideset{}{'}\sum\limits_{\chi\ \! \bmod{p}}$ is bounded, we deduce that
\begin{equation} \label{char22}
|E_1(p)|\ll
\frac{x^{1+\varepsilon_1}F(\alpha,\beta)}{p} \left(A\max\limits_{\substack{\chi\ \! \bmod{p}\\ \chi\not=\chi_0}} \left| \sum\limits_{|b|\le B} 
\chi(b) \right| +  
B\max\limits_{\substack{\chi\ \! \bmod{p}\\ \chi\not=\chi_0}} 
\left| \sum\limits_{|a|\le A}\chi(a) \right|
\right).
\end{equation}
Now using \eqref{char22} and Lemma \ref{gara} with $\eta=1/20$, we obtain
\begin{equation} \label{charbound}
\sum\limits_{\substack{p\le x\\ p\equiv 1\ \mbox{\scriptsize \rm mod}\ 4}} |E_1(p)| \ll
x^{1+2\varepsilon_1}F(\alpha,\beta)AB^{19/20}+
x^{1+2\varepsilon_1}F(\alpha,\beta)A^{19/20}B+
x^{19/20+\varepsilon}F(\alpha,\beta)AB.
\end{equation}

If $p\equiv 3$ mod $4$, then the term $E_1(p)$ can be written as a character sum similar to \eqref{char2} and be estimated by the same method. This leads to the same bound for
$ 
\sum\limits_{\substack{p\le x\\ p\equiv 3\ \mbox{\scriptsize \rm mod}\ 4}} |E_1(p)|
$
as \eqref{charbound}. Therefore, we obtain 
$$ 
|\mathcal{E}_1| \le \sum\limits_{p\le x} |E_1(p)| \ll
x^{1+2\varepsilon_1}F(\alpha,\beta)AB^{19/20}+
x^{1+2\varepsilon_1}F(\alpha,\beta)A^{19/20}B+
x^{19/20+\varepsilon}F(\alpha,\beta)AB
$$
which is
$$
\ll  \frac{xF(\alpha,\beta)}{\log^c x}
$$
for every fixed $c>0$ if $A,B\ge x^{\varepsilon}$, as desired. This completes the proof of Theorem 2.\\ \\
{\bf Acknowledgment.} The author wishes to thank W. Banks and I.E. Shparlinski for bringing their paper \cite{WDBIES} to my attention. He would further like to thank N.C. Jones and L. Zhao for useful discussions.

Stephan Baier\newline
School of Engineering and Science, Jacobs University Bremen \newline
P. O. Box 750561, Bremen 28725, Germany \newline
Email: {\tt s.baier@jacobs-university.de} \newline

\begin{thebibliography}{99}
\bibitem[1]{Bai} S. Baier, {\it The Lang-Trotter conjecture on average},
to appear in J. Ramanujan Math. Soc., arXiv:math.NT/0609095.
\bibitem[2]{BaZh} S. Baier, L. Zhao, {\it The Sato-Tate Conjecture on Average for Small Angles}, to appear in Trans. Am. Math. Soc., arXiv:math.NT/0608318.
\bibitem[3]{WDBIES} W. D. Banks, I. E. Shparlinski, {\it Sato-Tate,
	   cyclicity, and divisibility statistics on average for
	   elliptic curves of small height}, preprint, ArXiv:math.NT/0609144.
\bibitem[4]{CHT} L. Clozel, M. Harris, and R. Taylor, {\it Automorphy for some l-adic lifts of automorphic mod l Galois representations}, preprint, available at www.math.harvard.edu/$\sim$rtaylor.
\bibitem[5]{Co} A.C. Cojocaru, {\it Questions about the reductions modulo primes of an elliptic curve}, Proceedings of the 7-th conference of the Canadian Number Theory Association (Montreal, 2002), ed. E. Goren and H. Kisilevsky, CRM Proceedings and Lecture Notes, Vol. 36 (2004) 61-79. 
\bibitem[6]{DaP} C. David, F. Pappalardi, {\it Average Frobenius
Distributions of Elliptic Curves},
Int. Math. Res. Not. (1999) 165-183.
\bibitem[7]{Deu} M. Deuring, {\it Die Typen der Multiplikatorenringe
elliptischer Funktionenk\"orper}, Abh. Math. Sem. Hansischen Univ. 14 (1941)
197-272.
\bibitem[8]{FMu} E. Fouvry, M.R. Murty, {\it
On the distribution of supersingular
primes}, Canad. J. Math. 48 (1996) 81-104.
\bibitem[9]{gar} M.Z. Garaev, {\it Character sums in short intervals and the multiplication table modulo a large prime}, Monat. Math. 148 (2006) 127-138.
\bibitem[10]{HSBT} M.Harris, N. Shepherd-Barron and R. Taylor, {\it Ihara's lemma and potential automorphy}, preprint, available at www.math.harvard.edu/$\sim$rtaylor.
\bibitem[11]{HIEK} H. Iwaniec, E. Kowalski, {\it
Analytic Number Theory},
American Mathematical Society Colloquium Publications,
American Mathematical Society, vol. 53.
\bibitem[12]{JaYu} K. James, G. Yu, {\it Average Frobenius Distribution
of Elliptic Curves},  Acta Arith.  124  (2006), 79--100.
\bibitem[13]{Jon} N. Jones, {\it The constants in the Lang-Trotter
conjecture}, preprint (2006).
\bibitem[14]{LTr} S. Lang, H. Trotter, {\it Frobenius Distributions in GL$_2$
extensions}, Lecture Notes in Math. 504 (1976) Springer-Verlag, Berlin.
\bibitem[15]{Tate} J.T. Tate, {\it
Algebraic cycles and poles of zeta functions},
Arithmetical algebraic Geom., Harper and Row, New York, 1965.
\bibitem[16]{Tayl} R. Taylor, {\it Automorphy for some $l$-adic lifts of automorphic mod $l$ representations II}, preprint, available at www.math.harvard.edu/$\sim$rtaylor.
\end{thebibliography}
\end{document}